\documentclass[11pt]{amsart}
\usepackage{amscd}                                            
\usepackage[arrow,matrix]{xy}
\usepackage{graphicx}
\usepackage{hyperref}
\usepackage{comment}
\usepackage{xcolor}

\definecolor{brown}{RGB}{150,100,0}
 
\definecolor{purple}{RGB}{150,0,100}

\definecolor{grey}{RGB}{128,128,128}

\input xypic
\numberwithin{equation}{section}
\theoremstyle{plain}
\newtheorem{lemma}{Lemma}[section]
\newtheorem{proposition}[lemma]{Proposition}
\newtheorem{theorem}[lemma]{Theorem}
\newtheorem{corollary}[lemma]{Corollary}

\theoremstyle{definition}
\newtheorem{definition}[lemma]{Definition}
\newtheorem{remark}[lemma]{Remark}
\newtheorem{example}[lemma]{Example}

\begin{document}
\newcommand{\R}{{\mathbb R}}
\newcommand{\C}{{\mathbb C}}
\newcommand{\E}{{\mathbb E}}
\newcommand{\F}{{\mathbb F}}
\renewcommand{\O}{{\mathbb O}}
\newcommand{\Z}{{\mathbb Z}} 
\newcommand{\N}{{\mathbb N}}
\newcommand{\Q}{{\mathbb Q}}
\renewcommand{\H}{{\mathbb H}}
\newcommand{\T}{{\mathbf T}}

\newcommand{\Aa}{{\mathcal A}}
\newcommand{\Bb}{{\mathcal B}}
\newcommand{\Cc}{{\mathcal C}}    
\newcommand{\Dd}{{\mathcal D}}
\newcommand{\Ee}{{\mathcal E}}
\newcommand{\Ff}{{\mathcal F}}
\newcommand{\Gg}{{\mathcal G}}    
\newcommand{\Hh}{{\mathcal H}}
\newcommand{\Kk}{{\mathcal K}}
\newcommand{\Ii}{{\mathcal I}}
\newcommand{\Jj}{{\mathcal J}}
\newcommand{\Ll}{{\mathcal L}}    
\newcommand{\Mm}{{\mathcal M}}    
\newcommand{\Nn}{{\mathcal N}}
\newcommand{\Oo}{{\mathcal O}}
\newcommand{\Pp}{{\mathcal P}}
\newcommand{\Qq}{{\mathcal Q}}
\newcommand{\Rr}{{\mathcal R}}
\newcommand{\Ss}{{\mathcal S}}
\newcommand{\Tt}{{\mathcal T}}
\newcommand{\Uu}{{\mathcal U}}
\newcommand{\Vv}{{\mathcal V}}
\newcommand{\Ww}{{\mathcal W}}
\newcommand{\Xx}{{\mathcal X}}
\newcommand{\Yy}{{\mathcal Y}}
\newcommand{\Zz}{{\mathcal Z}}
\newcommand{\V}{{\mathcal V}}

\newcommand{\zt}{{\tilde z}}
\newcommand{\xt}{{\tilde x}}
\newcommand{\Ht}{\widetilde{H}}
\newcommand{\ut}{{\tilde u}}
\newcommand{\Mt}{{\widetilde M}}
\newcommand{\Llt}{{\widetilde{\mathcal L}}}
\newcommand{\yt}{{\tilde y}}
\newcommand{\vt}{{\tilde v}}
\newcommand{\Ppt}{{\widetilde{\mathcal P}}}
\newcommand{\bp }{{\bar \partial}}

\newcommand{\ad}{{\rm ad}}
\newcommand{\Om}{{\Omega}}
\newcommand{\om}{{\omega}}
\newcommand{\eps}{{\varepsilon}}
\newcommand{\Di}{{\rm Diff}}
\newcommand{\vol}{{\rm vol}}
\newcommand{\Pro}[1]{\noindent {\bf Proposition #1}}
\newcommand{\Thm}[1]{\noindent {\bf Theorem #1}}
\newcommand{\Lem}[1]{\noindent {\bf Lemma #1 }}
\newcommand{\An}[1]{\noindent {\bf Anmerkung #1}}
\newcommand{\Kor}[1]{\noindent {\bf Korollar #1}}
\newcommand{\Satz}[1]{\noindent {\bf Satz #1}}

\renewcommand{\a}{{\mathfrak a}}
\renewcommand{\b}{{\mathfrak b}}
\newcommand{\e}{{\mathfrak e}}
\renewcommand{\k}{{\mathfrak k}}
\newcommand{\pg}{{\mathfrak p}}
\newcommand{\g}{{\mathfrak g}}
\newcommand{\gl}{{\mathfrak gl}}
\newcommand{\h}{{\mathfrak h}}
\renewcommand{\l}{{\mathfrak l}}
\newcommand{\sm}{{\mathfrak m}}
\newcommand{\n}{{\mathfrak n}}
\newcommand{\s}{{\mathfrak s}}
\renewcommand{\o}{{\mathfrak o}}
\newcommand{\so}{{\mathfrak so}}
\renewcommand{\u}{{\mathfrak u}}
\newcommand{\su}{{\mathfrak su}}
\newcommand{\ssl}{{\mathfrak sl}}
\newcommand{\ssp}{{\mathfrak sp}}
\renewcommand{\t}{{\mathfrak t }}
\newcommand{\X}{{\mathfrak X}}

\newcommand{\pb}{{\mathbf p}}
\newcommand{\rk}{{\rm rk}}
\newcommand{\grad}{{\nabla}}

\newcommand{\Cinf}{C^{\infty}}
\newcommand{\la}{\langle}
\newcommand{\ra}{\rangle}
\newcommand{\half}{\scriptstyle\frac{1}{2}}
\newcommand{\p}{{\partial}}
\newcommand{\notsub}{\not\subset}
\newcommand{\iI}{{I}}               
\newcommand{\bI}{{\partial I}}      
\newcommand{\LRA}{\Longrightarrow}
\newcommand{\LLA}{\Longleftarrow}
\newcommand{\lra}{\longrightarrow}
\newcommand{\LLR}{\Longleftrightarrow}
\newcommand{\lla}{\longleftarrow}
\newcommand{\INTO}{\hookrightarrow}
\newcommand{\sppt}{{\mathrm  sppt}}

\newcommand{\QED}{\hfill$\Box$\medskip}
\newcommand{\UuU}{\Upsilon _{\delta}(H_0) \times \Uu _{\delta} (J_0)}
\newcommand{\bm}{\boldmath}

\newcommand{\er}{\mathbb{R}}
\newcommand{\ce}{\mathbb{C}}
\newcommand{\be}{\begin{equation}}
\newcommand{\bel}[1]{\begin{equation}\label{#1}}
\newcommand{\qe}{\end{equation}}
\newcommand{\ee}{\end{equation}}
\newcommand{\eeq}{\end{equation}}
\newcommand{\ba}{\begin{eqnarray}}
\newcommand{\ea}{\end{eqnarray}}
\newcommand{\rf}[1]{(\ref{#1})}


\title[The Fisher  metric and probabilistic mappings]{Diffeological statistical models, the Fisher metric   and  probabilistic  mappings}

\author[H. V. L\^e]{H\^ong V\^an L\^e}
\address{Institute of Mathematics, Czech Academy of Sciences, Zitna 25, 11567 Praha 1, Czech Republic}
\email{hvle@math.cas.cz}

\thanks{Research of HVL  was supported by RVO:67985840 and the GA\v CR-project 18-01953J}
\begin{abstract}
In this note   we  introduce  the notion of a $C^k$-diffeological statistical model, which allows us to apply the  theory  of diffeological spaces to   (possibly  singular)  statistical models.  In particular, we  introduce a class of  almost 2-integrable  $C^k$-diffeological statistical models  that  encompasses  all known statistical models    for which the Fisher metric is  defined. This class   contains a statistical  model which    does not  appear  in the  Ay-Jost-L\^e-Schwachh\"ofer theory of parametrized   measure  models.  Then we  show that  for any positive integer  $k$  the   class  of  almost 2-integrable $C^k$-diffeological  statistical models  is preserved under     probabilistic mappings.  Furthermore,   the monotonicity theorem  for the Fisher metric     also holds  for this class.  As  a consequence, the  Fisher  metric
	on an almost  2-integrable $C^k$-diffeological  statistical  model  $P \subset \Pp(\Xx)$  is preserved
	under   any probabilistic   mapping  $T: \Xx \leadsto \Yy$  that is sufficient  w.r.t. $P$.  Finally we extend the  Cram\'er-Rao inequality  to the   class of  2-integrable $C^k$-diffeological
	statistical models.
\end{abstract}
\maketitle

	
	
	\section{Introduction}
	In  mathematical statistics,  the  notion of {\it a   statistical model}  and  the notion  of   {\it a parametrized  statistical model}   are of central importance \cite{McCullagh2000}. 
	For a measurable   space $\Xx$,  let us denote by $\Pp(\Xx)$ the space of all probability measures
	on $\Xx$. 
	According to currently accepted theories, see  e.g. \cite{McCullagh2000} and references therein,
	a statistical
	model is a  subset $P_\Xx\subset \Pp(\Xx)$  and a parameterized
	statistical model is a parameter set $\Theta$ together with a mapping  $\pb: \Theta \to \Pp(\Xx)$.
	The  image  $\pb(\Theta)\subset  \Pp(\Xx)$  is   a statistical  model   endowed  with the parametrization  $\pb:  \Theta \to \pb(\Theta)$.  If   the  parameter  set $\Theta$  is a  smooth   manifold,  then we can  study  a statistical model $\pb(\Theta)$, endowed  with   a  parametrization $\pb: \Theta \to \pb(\Theta)\subset \Pp(\Xx)$, by  applying      differential  geometric  techniques  to $\Theta$ and   to smooth mappings  $\pb : \Theta \to \Pp(\Xx)$.  This idea   lies  in the  heart  of  the  field  {\it Information Geometry},  which    is the domain  of  mathematical statistics       where  we study (parameterized)    statistical models  using techniques of differential geometry \cite{Chentsov1972}, \cite{Amari1985}, \cite{Amari2016}, \cite{AJLS2017}.  
	In the book ``Information Geometry"   by Ay-Jost-L\^e-Schwachh\"ofer,  a    parameterized  statistical model   is a triple $(M, \Xx, \pb)$ where
	$M$ is a Banach manifold, $\Xx$ is a measurable   space,  and $i \circ \pb: M \stackrel{\pb}{\to} \Pp (\Xx)\stackrel{i}{\to} \Ss(\Xx)$ is a  $C^1$-map. Here   $\Ss
	(\Xx)$ is the Banach  space  of  all signed finite measures  on $\Xx$  endowed  with the  total variation  norm  $\| \cdot \|_{TV}$ and  $i$  is  the natural  inclusion.  We would like to emphasize  that the concept of a parametrized  statistical  model introduced in 
	\cite{AJLS2015, AJLS2017, AJLS2018} encompasses    statistical models endowed  with   the  structure  of   a finite  dimensional  manifold  \cite{Chentsov1972}, \cite{Amari1985, AN2000}, or   with  the structure  of  an infinite  dimensional    Banach manifold  \cite{PS1995}.   The  theory  of parametrized measure  models,  moreover, allows us    to study    {\it  singular  statistical models $P_\Xx$}   using   differential    geometric techniques,  if $P_\Xx$  is endowed  with a
	parameterization  by  a  Banach  manifold. 

	In this note, inspired by the theory of  diffeological spaces  founded  by Souriau  and developed further  by many people,    we shall  generalize  the concept  of a  parameterized   statistical model      to  the concept  of a
	{\it   $C^k$-diffeological statistical model}  $P \subset \Pp(\Xx)$ which, by definition,  is a subset in  $\Pp(\Xx)$   endowed  with  a {\it compatible $C^k$-diffeology}. We shall  show that  the  concept  of     a  $C^k$-diffeological  statistical  model  is more flexible  than   the concept  of a  parameterized  statistical model. In  particular,
	the image  $\pb (M)$ of any parameterized    statistical model  $(M, \Xx, \pb)$ has  a natural  compatible  $C^1$-diffeology. Moreover,  for any $k \in \N^+ \cup \infty$, any subset  in $\Pp(\Xx)$ can be provided  with a  compatible  $C^k$-diffeology (and hence it has a structure  of  a  $C^k$-diffeological    statistical model).  Furthermore,  not every  subset  in $\Pp(\Xx)$ can be written as $\pb (M)$  for some    parameterized    statistical model $(M, \Xx, \pb)$. Hence   the    class  of $C^1$-diffeological statistical models  is larger  than the class of    statistical  models   parameterized  by     Banach manifolds   as in Ay-Jost-L\^e-Schwachh\"ofer's theory.   We   also     extend  conceptually many    results  in Ay-Jost-L\^e-Schwachh\"ofer's theory  concerning  differential geometry of parametrized   statistical models and their application to   statistics  to  the class of $C^k$-diffeological statistical models, using the  theory  of probabilistic  mappings developed in
	a recent  work by   Jost-L\^e-Luu-Tran \cite{JLLT2019}. 
	
	Our note is organized as follows.  In the second   section we  introduce  the notions of $C^k$-diffeological statistical models,    almost 2-integrable  $C^k$-diffeological statistical models,  and  2-integrable    $C^k$-diffeological  statistical models. In the third  section  we recall the notion of  probabilistic mappings  and    related   results in \cite{JLLT2019}
	and   prove  that   the class  of  (almost 2-integrable/ resp. 2-integrable)  $C^k$-statistical  models   is  preserved  under  probabilistic mappings (Theorem \ref{thm:inv}).Then  we  extend  the monotonicity  of the Fisher metric on  2-integrable     parameterized  statistical models to the class  of   almost 2-integrable    $C^k$-diffeological statistical  models (Theorem \ref{thm:mono}). In the last section we  prove      a diffeological   version of the Cram\'er-Rao inequality (Theorem \ref{thm:cr})  which extends  previously known    versions  of the Cram\'er-Rao inequality in \cite{AJLS2017}, \cite{LJS2017b}.
	We  conclude    our  paper   with  a  discussion  on  some future    directions and open questions.

	\section{Almost 2-integrable diffeological statistical models}
	Given  a statistical model $P \subset \Pp (\Xx)$ which we  also denote by $P_\Xx$,  it is known that 
	$P_\Xx$ is endowed  with a {\it natural  geometric structure} induced   from  the Banach  space
	$(\Ss (\Xx), ||, || _{TV})$. 
	\begin{definition} 
		\label{def:tangent} (cf \cite[Definition 3.2, p. 141]{AJLS2017})
		(1) Let $(V, \| \cdot \|)$ be a Banach space,  $X \stackrel{i}{\INTO} V$  an arbitrary  subset,  where $i $ denotes the inclusion,  and $x_0 \in X$. Then $v \in V$  is  called { \it  a tangent vector of $X$  at $x_0$},
		if  there is  a $C^1$-map  $c: \R \to X $, i.e., the    composition  $i\circ  c : \R\to  V$  is a $C^1$-map,   such that $c(0) = x_0$ and $\dot c (0) = v$.   
		
		(2) {\it The  tangent (double) cone} $C_xX$  at a point $ x\in X$ is defined as the   subset  of the tangent   space $T_x V = V$    that consists  of  tangent  vectors  of $X$   at   $x$.  
		The {\it tangent  space} $T_xX$ is  the linear hull  of the
		tangent cone $C_xX$.
		
		(3)  {\it The tangent cone fibration $C X$}  (resp. {\it the tangent fibration $TX$}) is  the union
		$\cup_{x\in X}  C_x X$ (resp.  $\cup_{x\in X} T_xX$), which   is a  subset  of $V \times V$  and therefore it is  endowed with the  induced topology  from  $V \times V$.
	\end{definition}
	
	\begin{remark}\label{rem:comp} (1) The  notion  of a  tangent cone   in Definition \ref{def:tangent}   occurs in a similar fashion  in   the  theory  of singular  spaces,  see  e.g. \cite[\S 3]{LSV2013}, \cite[\S 3]{LSV2015},  \cite[p. 166]{IZ2013}.  
		
		(2) Definition \ref{def:tangent}   differs  from  \cite[Definition 3.1]{AJLS2017}  in that  in  Definition   \ref{def:tangent}   the domain  of a  $C^1$-curve $c$  is  $\R$ and    in \cite{AJLS2017}    the  domain of  a $C^1$-curve  $c$ is $(-\eps, \eps)$. Since
		$(-\eps, \eps)$  is diffeomorphic  to $\R$,  both the two choices  of  the domain of $c$  are   equivalent.
	\end{remark}
	
	\begin{example}\label{ex:mix}   Let  us consider  a   mixture family $P_\Xx$ of probability measures $p_\eta \mu_0$ on $\Xx$ that are dominated by $\mu_0 \in \Pp(\Xx)$, where the density  functions
		$p_\eta$ are of  the following form
		\begin{equation}\label{eq:mix}
		p_\eta(x): =  g^1(x)\eta_1 +  g^2 (x) \eta_1 + g^3 (x)(1-\eta_1 -\eta_2) \text{  for } x\in \Xx.
		\end{equation}
		Here $g^i$,  for $i=1,2,3$, are  nonnegative functions  on $\Xx$
		such that $\E_{\mu_0} (g^i)= 1$ and  $\eta=(\eta_1, \eta_2) \in D_b \subset \R^2$  is a parameter, which will be specified as follows.
		Let us divide   the square $D= [0,1] \times [0,1] \subset \R^2$    in smaller  squares   and color  them in black and white  like a chessboard.
		Let $D_b$  be  the  closure  of the subset of $D$ colored in   black.  If $\eta$ is an interior  point
		of $D_b$  then  $C_{p_\eta}P_\Xx = \R^2$.  If  $\eta$ is a boundary point of $D_b$ then  $C_{p_\eta} P_\Xx = \R$. If  $\eta$ is  a corner  point of $D_b$, then $C_{p_\eta}P_\Xx$ consists  of two intersecting  lines.
	\end{example}
	
	\
	
	$\bullet$ Let $P_\Xx$ be a statistical model.
	Then it is known  that  any $v \in  C_\xi P_\Xx$  is dominated by $\xi$. 
	Hence
	{\it  the logarithmic representation of $v$}  
	\begin{equation}
	\label{eq:logrep}
	\log v: =dv/d\xi
	\end{equation} 
	is an element  of $L^1(\Xx,\xi)$.
	The  set $\{ \log  v|\:  v\in C_\xi P_\Xx\}$  is a  subset  in
	$L^1(\Xx, \xi)$. We denote      it   by  $\log (C_\xi P_\Xx)$ and  will  call  it {\it the logarithmic representation  of $C_\xi P_\Xx$}.
	
	$\bullet$ Next  we want to put   a Riemannian metric on a statistical model $P_\Xx$  i.e., to put a  positive  quadratic  form  $\g$  on each {\it  tangent space  $T_\xi P_\Xx\subset  L^1(\Xx, \xi)$}. 
	The   space $L^1(\Xx, \xi)$ does not have a natural metric  but  its subspace $L^2(\Xx, \xi)$ is a Hilbert space.
	
	\begin{definition}
		\label{def:fisher}
		A statistical  model   $P_\Xx$   will be called {\it  almost 2-integrable}, if 
		\begin{equation}
		\log(C_\xi P_\Xx)   \subset L^2 (\Xx, \xi)\label{eq:logl2}
		\end{equation} 
		for  all $\xi \in P_\Xx$.
		In this case we define   the Fisher metric $\g$  on $P_\Xx$ as  follows. For each $v, w \in C_\xi P_\Xx$ 
		\begin{equation}
		\label{eq:rFisher}  \g_\xi(v, w): = \la \log v, \log w \ra _{ L^2 (\Xx, \xi)} = \int_\Xx\log v \cdot \log w \, d\xi.
		\end{equation}
		
	\end{definition}
	
	Since $T_\xi P_\Xx$  is the linear hull of $C_\xi P_\Xx$, the  formula (\ref{eq:rFisher})
	extends  uniquely  to a  positive quadratic   form on $T_\xi P_\Xx$, which is  called {\it the Fisher metric}.  
	

	\begin{example}\label{ex:classicfisher}
		\label{exer:fisher} Let us reconsider    Example \ref{ex:mix}.  Recall  that  our   statistical model  $P_\Xx$    is parameterized by  a map
		$$\pb:  D_b \to \Ss(\Xx), \eta\mapsto p_\eta \cdot \mu_0,$$
		which is  the restriction  of   the affine map  $L: \R^ 2 \to \Ss(\Xx)$, defined
		by the same formula. Hence  any   tangent  vector $\tilde v \in T_\eta P_\Xx$ can be written as $\tilde  v = d\pb  (v)$ where  $v \in T_\eta D_b$. Note that  for  $ v =  (v_1, v_2) \in  T_\eta D_b$ we have $d\pb ( v) = [ (g^1 - g^3) v_1 + (g^2- g^3)v_2]\mu_0$.  If $g^ i(x) > 0  $ for all $ x\in \Xx$ and $i = 1,2, 3$, then  $p_\eta (x) > 0$ for  all  $x\in \Xx$ and all $ \eta \in D_b$. Therefore
		$$\log d\pb (v)_{|\pb (\eta)}= \frac{ d \pb  (v)}{d (p_\eta \mu_0)}= \frac{(g^1 - g^3) v_1 + (g^2- g^3)v_2}{p_\eta}\in L^1 (\Xx,\pb (\eta)).$$
		Hence $P_\Xx$ is almost 2-integrable,  if  
		$$\frac{g^1 -g^3}{\sqrt p_\eta}, \frac{g^2 - g^3}{\sqrt p_\eta} \in L^2 (\Xx, \mu_0) \: \forall \eta \in D_b.$$
		In this case   we have
		\begin{equation}\label{eq:fisher2}
		\g_{| \pb (\eta)} (d\pb (v), d\pb(w)) = \la \log d\pb (v), \log d\pb(w) \ra _{L^2(\Xx,\pb(\eta))}.
		\end{equation}
	\end{example}

	Next we shall introduce the notion   of      a $C^k$-diffeological   statistical model.  
	\begin{definition}\label{def:diff}  
		For $k \in\N^+ \cup \infty$  and   a nonempty set $X$,   a {\it $C^k$-diffeology} of $X$ is a set $\Dd$ of   mappings  $\pb: U \to  X$,  where $U$ is an open domain  in $\R^n$, and $n$   runs over  nonnegative integers,  such that the three following
		axioms are satisfied.
		
		D1. {\it Covering}. The set $\Dd$ contains the constant mappings ${\bf x}:   r\mapsto x$, defined on $\R^n$, for all $x \in X$  and  for all $n \in \N$.
		
		D2. {\it  Locality}. Let ${\pb}: U \to X$ be a mapping. If for every
		point $r \in U$ there exists  an open neighborhood  $V $ of $r$ such that  ${\pb}_{| V}$ belongs to  $\Dd$ then  the map $\pb$ belongs to $\Dd$.
		
		D3. {\it Smooth compatibility}. For every element $\pb: U \to  X$ of $\Dd$, for every real
		domain $V$, for every  $\psi \in C^k(V, U)$, $\pb \circ \psi$ belongs to $\Dd$.
		
		A {\it $C^k$-diffeological space} is a nonempty set equipped with a $C^k$-diffeology  $\Dd$. Elements  $\pb: U \to X$  of $\Dd$   will be called {\it  $C^k$-maps  from $U$ to $X$}.
		
		A   statistical model  $P_\Xx$  endowed
		with a $C^k$-diffeology $\Dd_\Xx$  will be called  a {\it $C^k$-diffeological statistical  model},  if   for any  map $\pb: U \to P_\Xx$     in $\Dd_\Xx$ the composition $i \circ \pb: U \to  \Ss(\Xx)$ is  a $C^k$-map.  
	\end{definition}
	\begin{remark}\label{rem:diffeo} (1) In \cite{IZ2013}  Iglesias-Zemmour  considered  only  $C^\infty$-diffeologies.  The notion  of a $C^k$-diffeology   given  in Definition \ref{def:diff}  is a straightforward adaptation  of the  concept of    a smooth diffeology
		given in  \cite[\S 1.5]{IZ2013}.
		
		(2) Since  $(\Ss(\Xx), \|\cdot\|_{TV})$ is a Banach space, by
		\cite[Lemma 3.11, p. 30]{KM1997}, a compatible $C^\infty$-diffeology
		on a statistical model $P_\Xx$  is defined by  smooth maps $c: \R \to P_\Xx$.
		
		(3)  Given  a  $C^k$-diffeological  statistical  model  $(P_\Xx, \Dd_\Xx)$  and $\xi \in P_\Xx$,  the tangent cone $C_\xi (P_\Xx, \Dd_\Xx)$  is  the subset  of  $C_\xi P_\Xx$ that consists  of  the tangent  vectors $\dot  c(0)$  of   
		$C^k$-curves $c : \R \to \Xx$  in $\Dd_\Xx$ such that   $c(0)  = \xi$.  Similarly,
		the tangent  space $T_\xi (P_\Xx, \Dd_\Xx)$ is the linear  hull of $C_\xi (P_\Xx, \Dd_\Xx)$.
		
		(4) 
		Let  $(P_\Xx, \Dd_\Xx)$ be a $C^k$-diffeological  statistical model  and $V$  a   locally convex  vector  space. A  map  $\varphi: P_\Xx \to V$ is called {\it Gateaux-differentiable} on
		$(P_\Xx, \Dd_\Xx)$  if   for  any  $C^k$-curve  $c: \R \to P_\Xx$ in $\Dd_\Xx$  the composition $\varphi \circ  c: \R \to V$  is   differentiable.  We recommend \cite{KM1997}   for
		differential calculus  on locally   convex vector spaces.
	\end{remark}
	\begin{example}\label{ex:para}  (1)   Let $(M, \Xx, \pb)$ be a parametrized     statistical model.	Then  $(\pb (M), \Dd_\Xx)$  is  a $C^1$-diffeological statistical model where $\Dd_\Xx$ consists  of all  
		$C^1$-maps  ${\bf q}: \R^n \supset U  \to  \pb (M)$  such that    there exists a $C^1$-map  $\psi^M: U \to M$ and ${\bf q} = \pb \circ \psi^M$.  
		
		(2)  Let $P_\Xx$  be  a statistical model.  Then   $P_\Xx$   can be endowed   with a structure of   a  $C^k$-diffeological statistical model  for  any $k \in \N^+ \cup \infty$,  where its diffeology  $\Dd_\Xx ^{(k)}$   consists  of all mappings ${\pb}: U \to P_\Xx$   such that
		the composition $ i\circ \pb: U \to \Ss(\Xx)$  is of the class  $C^k$, where  $U$ is any open domain in  $\R^n$  for  $n \in \N$.
		
		(3)   Let  $\Xx$ be the closed interval $[0,1]$.   Let  $P_\Xx : =  f \cdot \mu_0$, where $f \in C^\infty (\Xx)$  such that $\int _\Xx f d\mu_0 =1$ and $f(x) > 0$  for  all $x\in \Xx$.
		We claim  that,    there   does not exist  a  parameterized   statistical model
		$(M, \Xx, \pb)$ such that 
		$P_\Xx  =\pb(M)$. Assume the opposite, i.e., 
		there  is a  $C^1$-map  $\pb:  M \to  \Ss(\Xx)$ such that  $\pb  (M)  =  P_\Xx$.
		Then  for any  $m \in  M$ we have  $d\pb ( T_m (M)) =  T_{\pb (m)}P_\Xx = \{ f\in C^\infty(\Xx)|\, \int_\Xx f\, d\mu_0 = 0\}$.  But this is not the case,   since    it is known  that  the space
		$C^\infty ([0,1])$  cannot be the image of   a  linear  bounded  map from  a Banach  space $M$
		to    $L_1 ([0,1])$, see e.g.  \cite[p. 1434]{Grabiner1974}. 
	\end{example}
	
	\begin{definition}\label{def:2-integrable}  
		A $C^k$-diffeological  statistical model $(P_\Xx, \Dd_\Xx)$  will be  called {\it almost 2-integrable}, if
		$\log (C_\xi  (P_\Xx, \Dd_\Xx)) \subset   L^2 (\Xx, \xi)$ for all $\xi \in P_\Xx$.
		
		An almost 2-integrable  $C^k$-diffeological  statistical model $(P_\Xx, \Dd_\Xx)$ will be  called {\it 2-integrable},  if   for  any $C^k$-map $\pb: U \to P_\Xx$ in $\Dd_\Xx$,
		the function $v \mapsto |d\pb(v)|_\g$ is continuous  on $TU$. 
	\end{definition}
	
	\begin{example}\label{ex:2-integr} (1)  By \cite[Theorem 3.2, p. 155]{AJLS2017}, a parameterized  statistical model  $(M, \Xx, \pb)$ is 2-integrable, iff and  only if  $(\pb(M), \pb_*(\Dd_M))$ is a   2-integrable  $C^1$-diffeological   statistical model. 
		
		(2)  The  $C^1$-diffeological statistical  model $(P_\Xx, \Dd_\Xx^{(1)})$ in Example \ref{ex:para}(3)  is 2-integrable, though
		there  is no  parameterized  statistical model  $(M, \Xx, \pb)$ such that $\pb (M) = P_\Xx$.

		(3)   Let $\Xx$ be a measurable space   and   $\lambda$  be    a $\sigma$-finite measure. In \cite[p. 274]{Friedrich1991}  Friedrich considered
		a  family  $P(\lambda):=\{  \mu \in \Pp(\Xx)|\, \mu \ll   \lambda\}$  that is endowed  with   the following  diffeology  $\Dd (\lambda)$.   A  curve $c: \R \to  P (\lambda)$ is a   $C^1$-curve, iff    
		$$\log\dot  c (t) \in  L^2 (\Xx, c(t)).$$
		Hence $(P(\lambda),\Dd(\lambda))$ is     an  almost 2-integrable
		$C^1$-diffeological statistical model. 
	\end{example}
	
	\begin{remark}\label{rem:his}  The axiomatics of Espaces différentiels, which became later the diffeological spaces, were introduced by J.-M. Souriau in the beginning of the eighties \cite{Souriau1980}. Diffeology is a variant of the theory of differentiable spaces, introduced
		and developed a few years before by K.T. Chen \cite{Chen1977}.
		As  I have     worked  with a  different  theory of smooth structures  on  singular spaces \cite{LSV2013, LSV2015},  I appreciate  the elegance  of the theory of diffeology     for its   consistent  and simple    treatment  of   smooth  structures on (possibly infinite dimensional) singular  spaces.  The  best source for diffeology
		is the monograph  by P. Iglesias-Zemmour \cite{IZ2013}.
	\end{remark}
	
	\section{Probabilistic mappings} \label{sec:prop}
	

	In 1962 Lawvere   proposed  a categorical approach to     probability theory,  where  morphisms  are  Markov kernels, and   most importantly, he supplied  the space $\Pp(\Xx)$  with  a natural  $\sigma$-algebra   $\Sigma_w$, making the notion of Markov kernels  and hence
	many  constructions  in  probability theory  and  mathematical statistics functorial.

	Let me recall the definition of $\Sigma_w$.
	Given a measurable space $\Xx$,  let $\Ff_s(\Xx)$ denote
	the  linear space of simple functions on $\Xx$. Recall that $\Ss(\Xx)$ is the  space of all signed finite  measures on $\Xx$. There  is a natural  homomorphism $I: \Ff_s(\Xx) \to \Ss^*(\Xx):= Hom (S(\Xx), \R), \, f \mapsto  I_f$,
	defined  by  integration:  $I_f (\mu):= \int_\Xx f d\mu$  for $f \in \Ff_s(\Xx)$ and $\mu \in \Ss(\Xx)$. 
	Following Lawvere  \cite{Lawvere1962},  we  define $\Sigma_w$ to be  the smallest $\sigma$-algebra  on $\Ss(\Xx)$  such that $I_f$  is measurable  for all $f\in \Ff_s (\Xx)$.   Let  $\Mm(\Xx)$ denote  the space  of all finite  nonnegative measures on $\Xx$. We also  denote  by $\Sigma_w$  the restriction  of  $\Sigma_w$
	to $\Mm (\Xx)$, $\Mm^* (\Xx): = \Mm(\Xx) \setminus \{0\}$,  and  $\Pp (\Xx)$.

	$\bullet$ For a   topological  space $\Xx$   we  shall  consider the natural  Borel  $\sigma$-algebra  $\Bb(\Xx)$.   Then  every  continuous  function is measurable  wrt $\Bb(\Xx)$.  Note that if $\Xx$  is  moreover  a  metric  space then $\Bb(\Xx)$   is the smallest  algebra making measurable  any   continuous function (\cite[Lemma 2.13]{Bogachev2018}).

	$\bullet$  Let  $C_b(\Xx)$ be the space 
	of {\it bounded  continuous functions}  on  a topological space $\Xx$. We denote by $\tau_v$ 
	the smallest topology on  $\Ss(\Xx)$  such that 
	for any  $f\in C_b (\Xx)$    the map $I_f: (\Ss(\Xx), \tau_v)\to \R$ is continuous. We also denote by
	$\tau_v$ the     restriction  of $\tau_v$  to $\Mm(\Xx)$ and $\Pp(\Xx)$, which  is also called  the  {\it weak topology}  that  generates
	the weak convergence of probability measures. 
	It is known that $(\Pp(\Xx), \tau_v)$ is separable, metrizable  if and only if  $\Xx$ is  \cite[Theorem 3.1.4, p. 104]{Bogachev2018}. 
	If $\Xx$ is separable and metrizable then  the Borel $\sigma$-algebra on $\Pp(\Xx)$ generated by $\tau_v$  coincides with $\Sigma_w$.

	\begin{definition}\label{def:prob} (\cite[Definition 2.4]{JLLT2019}) {\it A probabilistic  mapping} (or  {\it an arrow}) 
		from   a measurable space $\Xx$ to a measurable  space $\Yy$  is a
		measurable  mapping  from  $\Xx$ to $ (\Pp (\Yy), \Sigma_w)$.
	\end{definition}
	
	We shall  denote by $\overline{T}: \Xx \to  (\Pp(\Yy), \Sigma_w)$ the measurable  mapping  defining/generating a probabilistic  mapping $T: \Xx \leadsto \Yy$. 
	Similarly, for a measurable mapping $\pb: \Xx \to \Pp(\Yy)$  we shall denote by $\underline{\pb}: \Xx \leadsto \Yy$ the generated probabilistic  mapping.   Note that    a   probabilistic  mapping  is  denoted  by a curved  arrow  and  a measurable mapping by a straight arrow.

	\begin{example}\label{ex:prob} (\cite[Example 2.6]{JLLT2019})	(1) Assume that $\Xx$ is separable and metrizable. Then the  identity mapping $Id_\Pp:(\Pp(\Xx), \tau_v) \to (\Pp (\Xx), \tau_v)$ is  continuous, and hence 
		measurable w.r.t.  the Borel $\sigma$-algebra $\Sigma_w = \Bb(\tau_v)$. Consequently $Id_\Pp
		$  generates    a probabilistic mapping $ev: (\Pp (\Xx), \Bb(\tau_v))  \leadsto  (\Xx, \Bb(\Xx))$  and we write  $\overline {ev} = Id_\Pp$.  Similarly, for any measurable space $\Xx$,   we also  have   an    arrow (a probabilistic mapping)  $ev: (\Pp(\Xx), \Sigma_w) \leadsto \Xx$  generated  by     the  measurable mapping $\overline {ev} = Id_\Pp$.
		
		(2)  Let $\delta_x$ denote   the Dirac  measure concentrated at $x$. It is known that the    map $\delta: \Xx \to (\Pp (\Xx),\Sigma_w), \: x \mapsto \delta (x) : = \delta_x$,  is measurable \cite{Giry1982}. If $\Xx$  is a topological space, then the map
		$\delta: \Xx \to (\Pp(\Xx), \tau_v)$  is continuous, since  the composition $I_f \circ \delta: \Xx \to \R$ is continuous  for any  $f \in C_b (\Xx)$. 
		Hence,  if $\kappa: \Xx \to \Yy$  is a measurable  mapping  between measurable  spaces (resp.  a continuous mapping between  separable  metrizable  spaces),   then   the  map
		$\overline \kappa: \Xx \stackrel{\delta\circ \kappa}{\to } \Pp (\Yy)$  is a measurable mapping (resp. a continuous mapping).   We  regard   $\kappa$ as a probabilistic  mapping defined  by   $\delta \circ \kappa: \Xx \to \Pp(\Yy)$.
		In particular,   the identity  mapping   $ Id : \Xx \to \Xx$  of a measurable   space $\Xx$   is a probabilistic mapping generated by   $ \delta:  \Xx  \to \Pp(\Xx)$.  Graphically speaking,  any straight  arrow (a  measurable mapping)  $\kappa: \Xx \to \Yy$ between  measurable  spaces  can be  seen as  a curved arrow (a probabilistic  mapping).
	\end{example}

	Given  a    probabilistic  mapping   $T:  \Xx \leadsto \Yy$,   we define   a   linear  map   $S_*(T): \Ss(\Xx) \to \Ss(\Yy)$, called {\it Markov morphism}, 
	as follows  \cite[Lemma 5.9, p. 72]{Chentsov1972}
	\begin{equation}\label{eq:markov1}
	S_*(T) (\mu) (B): = \int_{\Xx}\overline T (x) (B)d\mu (x)
	\end{equation}
	for any    $\mu \in \Ss (\Xx)$  and  $B \in \Sigma_\Yy$.

	
	\begin{proposition}\label{prop:markov} Assume that  $T: \Xx  \leadsto \Yy$  is a probabilistic   mapping. 
		
		(1) Then $T$ induces   a linear bounded    map  $S_*(T): \Ss(\Xx) \to \Ss(\Yy)$ w.r.t. the total variation norm $|| \cdot || _{TV}$. The restriction $M_*(T)$  
		of $S_*(T)$ to  $\Mm(\Xx)$  (resp. $P_*(T)$  of $S_*(T)$  to  $\Pp (\Xx)$) maps  $\Mm(\Xx)$ to  $\Mm(\Yy)$ (resp. $\Pp (\Xx)$ to
		$\Pp(\Yy)$).

		(2) Probabilistic mappings   are morphisms in the category of  measurable  spaces,  i.e.,
		for any  probabilistic mappings $T_1: \Xx \leadsto \Yy$ and $ T_2: \Yy \leadsto \Zz$ we have
		\begin{equation}\label{eq:fm}
		M_*(T_2 \circ T_1)= M_*(T_2)\circ M_*(T_1), \: P_*(T_2 \circ T_1)= P_*(T_2)\circ P_*(T_1).
		\end{equation}
		
		(3)	$M_*$ and  $P_*$ are  faithful functors. 
		
		(4)  If $\nu \ll \mu\in  \Mm^*(\Xx)$ then  $M_*(T)(\nu)  \ll M_*(T) (\mu)$.
	\end{proposition}
	
	\begin{remark}\label{rem:markov} The first  assertion of Proposition \ref{prop:markov}  is due
		to Chentsov \cite[Lemma 5.9, p.72]{Chentsov1972}.  
		The second assertion  has been proved in \cite[Theorem 2.14 (1)]{JLLT2019},  extending      Giry'  result   in \cite{Giry1982}. The third assertion  has been proved in  \cite{JLLT2019}. 
		The last assertion   of  Proposition  \ref{prop:markov} is due  to Morse-Sacksteder \cite[Proposition 5.1]{MS1966}.
	\end{remark}
	
	We also denote by $T_*$  the map  $S_*(T)$ if no  confusion can arise.
	
	Given  a     probabilistic mapping $ T: \Xx \leadsto \Yy$  and a   $C^k$-diffeological  statistical
	model $(P_\Xx, \Dd_\Xx)$ we define a    $C^k$-diffeological  space  $(T_*(P_\Xx),  T_*(\Dd_\Xx))$    as the image of $\Dd$  by $T$ \cite[\S 1.43, p. 24]{IZ2013}. In other words, a mapping $\pb: U \to T_*(P_\Xx)$  belongs  to $T_*(\Dd_\Xx)$ if and only if it satisfies the following condition.  For every $r \in U$  there exists  an open
	neighborhood $V\subset U$ of $r$   such that  either  $\pb_{| V}$ is a constant mapping, or
	there exists  a  mapping  ${\bf q}: U \to  P_\Xx$  in $\Dd_\Xx$  such that  $\pb_{| V}  = T_* \circ {\bf q}$.

	\begin{theorem}\label{thm:inv}   Let  $T: \Xx  \leadsto \Yy$ be a probabilistic mapping  and
		$(P_\Xx, \Dd_\Xx)$  is a $C^k$-diffeological   statistical   model. 
		
		(1) Then  $(T_*(P_\Xx), T_*(\Dd_\Xx))$ is  a   $C^k$-diffeological  statistical model.
		
		(2)	If $(P_\Xx, \Dd_\Xx)$ is an almost 2-integrable  $C^k$-diffeological    statistical model,   then 
		$(T_* (P_\Xx), T_*(\Dd_\Xx))$  is  also an almost 2-integrable  $C^k$-diffeological  statistical model.
		
		(3)  	If $(P_\Xx, \Dd_\Xx)$ is a 2-integrable    $C^k$-diffeological  statistical model,   then 
		$(T_*(P_\Xx), T_*(\Dd_\Xx))$  is  also a  2-integrable  $C^k$-diffeological   statistical model.
		
	\end{theorem}

	\begin{proof}  (1) The first  assertion is straightforward, since $T_*: \Ss(\Xx) \to \Ss(\Yy)$ is a linear bounded map by Proposition \ref{prop:markov}(1).
		
		(2) 	 Assume that  $(P_\Xx, \Dd_\Xx)$ is an almost  2-integrable $C^k$-statistical model  and  $v \in C_\xi (P_\Xx, \Dd_\Xx)$.  Then there exits a  $C^k$-map  $c : \R \to  P_\Xx$ in $\Dd_\Xx$   such that  $\frac{d}{dt}_{|t =0} c (\xi) = v$.  Since    $T_*: \Ss(\Xx) \to  \Ss(\Yy)$ is a bounded linear map,    
		$$\frac{d}{dt}_{|t =0}T_* \circ  c = T_* (v).$$
		By the monotonicity  theorem \cite[Corollary 5.1, p. 260]{AJLS2017},
		we  have
		\begin{equation}
		\label{eq:mono}
		\| \frac{d  T_*  v}{d T_*\xi}\|_{L^2(\Yy, T_* \xi)} \le  \| v \|_{L^2(\Xx, \xi)}.
		\end{equation}
		This  proves  that  $(T_*(P_\Xx), T_*(\Dd_\Xx))$ is  almost  2-integrable.
		
		(3)    Assume that  $(P_\Xx, \Dd_\Xx)$ is a  $C^k$-diffeological  statistical  model.
		Let $c: \R \to T_*(P_\Xx)$  be an element  in $T_*(\Dd_\Xx)$.  Then
		$c= T_* \circ  c'$, where  $ c: \R \to P_\Xx$ is an  element  of $\Dd_\Xx$, i.e.,  $i \circ  c: \R \to \Ss(\Xx)$ is  of class  $C^k$ and 
		$(\R, \Xx, c)$ is a  parameterized  2-integrable   statistical model.   By \cite[Theorem 5.4, p. 264]{AJLS2017},  $(\R, \Yy,  T_* \circ c)$  is  a 2-integrable   parameterized  statistical model.  Combining with the   first assertion of Theorem \ref{thm:inv}   this proves  the last  assertion  of Theorem \ref{thm:inv}.
		
	\end{proof}

	Denote by $L(\Xx)$ the space  of  bounded  measurable functions on  a measurable space  $\Xx$.
	Given  a    probabilistic  mapping   $T:  \Xx \leadsto \Yy$,   we define   a   linear  map
	$T^*: L(\Yy) \to L(\Xx)$   as follows \cite[(2.2)]{JLLT2019}
	
	\begin{equation}\label{eq:adjoint1}
	T ^*(f)  (x): = I_f (\overline T (x))  = \int _\Yy f d  \overline T (x),
	\end{equation}
	which coincides  with the  classical formula  (5.1)  in \cite[p. 66]{Chentsov1972} for  the transformation  of  a  bounded measurable  $f$ under  a Markov morphism (i.e.,   a probabilistic mapping)  $T$.
	In particular, if $\kappa: \Xx \to \Yy$ is a measurable mapping,  then   we have $ \kappa^*(f) (x) = f(\kappa(x))$, since  $\overline  \kappa = \delta\circ \kappa$.

	\begin{definition}\label{def:suff}(\cite[Definition 2.22]{JLLT2019}, cf. \cite{MS1966}) Let $P_\Xx \subset \Pp(\Xx)$  and  $P_\Yy \subset \Pp(\Yy)$.
		A  probabilistic  mapping $T: \Xx \leadsto \Yy$    will be called {\it sufficient  for $P_\Xx$}  if  there exists a  probabilistic
		mapping $\underline \pb : \Yy \leadsto \Xx$ such that  for all $\mu \in P_\Xx$   and  $h \in L(\Xx)$ we have
		\begin{equation}
		\label{eq:suff}
		T_*(h \mu) = \underline{\pb}^*(h)T_*(\mu) \text{, i.e., }  \underline{\pb} ^* (h) = \frac{d T_* (h \mu)}{d T_* (\mu)}\in  L^1(\Yy, T_*(\mu)).
		\end{equation}
		In this case 
		we  shall call  the  measurable  mapping $\pb: \Yy \to \Pp(\Xx)$  defining    the probabilistic  mapping $\underline {\pb}: \Yy \leadsto \Xx$
		{\it a conditional  mapping   for  $T$}.
	\end{definition}

	\begin{example}
		\label{ex:suff}
		Assume  that $\kappa:\Xx \leadsto \Yy$  is a   measurable mapping  (i.e., a statistic) which is   a   probabilistic mapping sufficient for  $P_\Xx \subset \Pp(\Xx)$.   Let   $\pb:  \Yy \to \Pp (\Xx), \,  y \mapsto  \pb _y,$ be a conditional  mapping for $\kappa$. By (\ref{eq:adjoint1}),  $\underline\pb^* (1_A)(y)  =  \pb_y(A)$,  and  we  rewrite   (\ref{eq:suff})
		as follows
		\begin{equation}
		\label{eq:suffs}
		\pb_y(A)=  \frac{d\kappa_*(1_A\mu)}{d\kappa_*\mu}\in L^1(\Yy, \kappa_*(\mu)) .
		\end{equation}
		The RHS of (\ref{eq:suffs}) is  the conditional  measure of $\mu$  applied to  $A$  w.r.t. the  measurable mapping $\kappa$.  The   equality (\ref{eq:suffs})  implies  that this   conditional  measure is regular  and independent of  $\mu$. 
		Thus  the notion of sufficiency of a  measurable mapping $\kappa$  for  $P_\Xx$  coincides with the classical notion of   sufficiency  of      $\kappa$  for  $P_\Xx$, see e.g.,  \cite[p. 28]{Chentsov1972}, \cite[Definition 2.8, p. 85]{Schervish1997}.  We   also note that the equality  in (\ref{eq:suffs})  is understood  as  equivalence class  in $L^1(\Yy,  \kappa_*(\mu))$ and hence    every
		statistic $\kappa '$  that  coincides  with a    sufficient statistic   $\kappa$   except  on a zero $\mu$-measure   set,   for all  $\mu \in P_\Xx$,  is also   a sufficient  statistic  for  $P_\Xx$.
	\end{example}

	\begin{example}
		\label{ex:FN} (cf. \cite[Lemma 2.8, p. 28]{Chentsov1972}) Assume  that  $\mu \in \Pp(\Xx)$ has a regular  conditional    distribution    w.r.t.   to a statistic 
		$\kappa: \Xx \to \Yy$, i.e., there  exists  a measurable  mapping
		$\pb : \Yy \to \Pp(\Xx), y \mapsto \pb_y,$  such that
		\begin{equation}\label{eq:cond1}
		\E_\mu ^{\sigma(\kappa)}(1_A | y) = \pb_y(A) 
		\end{equation}
		for any $A \in \Sigma_\Xx$  and  $y \in \Yy$.  
		Let $\Theta$ be a set  and   $P : =\{\nu_\theta \in \Pp(\Xx) | \, \theta \in \Theta\}$  be   a  parameterized family of  probability measures  dominated  by $\mu$. If  there exist  a function
		$h : \Yy \times \Theta \to \R$  such  that  for all $\theta \in \Theta$  and  we have
		\begin{equation}\label{eq:NF}
		\nu_\theta = h(\kappa (x))\mu
		\end{equation}
		then $\kappa $ is sufficient  for  $P$, since    for any $\theta \in \Theta$ 
		$$\pb^* (1_A) = \frac{d\kappa_*  (1 _A  \nu_\theta)}{d\kappa_* \nu_\theta} $$
		does not depend on $\theta$.  The  condition (\ref{eq:NF})  is the Fisher-Neymann sufficiency condition  for a family  of  dominated
		measures.   
	\end{example}
	
	\begin{example}
		\label{ex:inj}  Let   $  \kappa: \Xx \to \Yy$ be a  measurable 1-1  mapping. Then   for  any    statistical  model $P_\Xx \subset \Pp(\Xx)$   the  statistic  $\kappa$ is sufficient  w.r.t.  $\Pp_\Xx$, since   for any  $A \in \Sigma_\Xx$  and any
		$\mu \in \Pp_\Xx$ we have
		$$\frac{d\kappa_*(1_A \mu)}{d\kappa_*\mu}=  (\kappa ^{-1})^* (1_A) \in L^1 (\Yy, \kappa_*(\mu)).$$
	\end{example}

	\
	
	
	Next we shall show that   probabilistic mappings don't increase  the Fisher metrics  on almost 2-integrable $C^k$-diffeological statistical models. Thus the Fisher  metric     serves  as  a ``information quantity"    of almost 2-integrable $C^k$-diffeological statistical models.
	
	\begin{theorem}\label{thm:mono}  Let  $T: \Xx  \leadsto \Yy$ be a probabilistic  mapping  and
		$(P_\Xx, \Dd_\Xx)$ an almost 2-integrable  $C^k$-diffeological statistical model.   Then
		for   any  $\mu \in P_\Xx$  and  any $v \in T_\mu (P_\Xx, \Dd_\Xx)$  we have
		$$  \g_\mu  (v, v) \ge \g _{T_*\mu}  (T_*v, T_*v)$$
		with the equality    if  $T$ is sufficient  w.r.t.   $P_\Xx$.	
	\end{theorem}
	
	\begin{proof}  The monotonicity  assertion  of Theorem \ref{thm:mono}  follows  from  (\ref{eq:mono}). The second
		assertion  of  Theorem \ref{thm:mono}  follows   from the first   assertion,  taking into account Theorem 2.8.2 in \cite{JLLT2019}  that  states   the  existence  of a    probabilistic  mapping
		$\pb: \Yy \leadsto \Xx$   such that   $\pb  _*  (T_ *(P_\Xx))  = P_\Xx$, and  therefore $\pb_* (T_* (\Dd_\Xx)) = \Dd_\Xx$.
	\end{proof}
	
	Let us    apply   Theorem \ref{thm:mono} to  Example  \ref{ex:2-integr} (3)  originally from \cite{Friedrich1991}.   In \cite[Satz 1, p.274]{Friedrich1991}  Friedrich  considered  the  group
	$\Gg (\Xx, \Sigma_\Xx, \lambda)$ of all  measurable  1-1 mappings
	$\Phi: \Xx \to \Xx$  such that $\Phi_*(\lambda)\ll  \lambda$.  Clearly
	$\Phi_* (P (\lambda)) \subset   P (\lambda)$.   Example \ref{ex:inj} says   that  $\Phi$  is a sufficient     statistic   w.r.t. $P(\lambda)$.  Hence  Theorem  \ref{thm:mono}  implies the  following
	\begin{corollary}
		\label{cor:fried}  (\cite[Satz 1]{Friedrich1991})  The group
		$\Gg (\Xx, \Sigma_\Xx, \lambda)$ acts  isometrically  on  $P(\lambda)$.
	\end{corollary}
	
	\begin{remark}\label{rem:mon} Theorem  \ref{thm:mono}  extends  the  Monotonicity  Theorem \cite[Theorem 5.5, p. 265]{AJLS2017}
		for 2-integrable  parameterized    statistical models.  \footnote{As  we  remarked in  Section \ref{sec:dis}, Theorem  \ref{thm:mono} can be  easily extended   to the  case of  almost $l$-integrable  $C^k$-diffeological measure  models.}  
	\end{remark}
	
	\section{The  Cram\'er-Rao inequality for  2-integrable diffeological  statistical models}\label{sec:cr}
	
	In this   section we  shall   prove  a  version of  the Cram\'er-Rao inequality   for  estimators  with values in a    2-integrable   $C^k$-diffeological  statistical   model.
	
	\begin{definition}
		\label{def:est}  Let   $P_\Xx\subset \Pp(\Xx)$  be a  statistical model. An {\it estimator}   is  a map  $\hat \sigma: \Xx \to P_\Xx$.
	\end{definition}
	
	Assume that  
	$V$ is a  locally convex  topological  vector  space. Then   we    denote by  $Map(P_\Xx, V)$  the space of all    mappings $\varphi:P_\Xx\to V$  and by $V'$  the topological   dual of $V$.
	It is usually  easier  to  estimate  only  a ``coordinate"  $\varphi(\xi)$ of a   probability measure $\xi \in P_\Xx$, which determines $\xi$  uniquely if  $\varphi$ is an embedding.
	
	\begin{definition}
		\label{def:vares}  Let  $P_\Xx$ be  a   statistical model  and $\varphi \in Map (P_\Xx, V)$.
		{\it  A $\varphi$-estimator $\hat \sigma_{\varphi}$}  is  a  composition  $\varphi\circ \hat \sigma:  \Xx \stackrel{\hat{\sigma}}{\to} P_\Xx \stackrel{\varphi} {\to} V$. 
	\end{definition}
	
	\begin{example}\label{ex:mke}   Assume that $k : \Xx \times \Xx  \to \R$ is a symmetric and positive definite kernel
		function  and $V$ be the  associated RKHS. For any $x\in \Xx$ we  denote  by  $k_x$  the function on $\Xx$ defined by $k_x (y):=  k(x, y)$  for any $y \in \Xx$.
		Then $k_x$ is an   element   of $V$. Let  $P_\Xx = \Pp(\Xx)$. Then  we define
		the {\it   kernel mean embedding} $\varphi: \Pp (\Xx) \to V$  as follows \cite{MFSS2017}
		$$ \varphi(\xi) : = \int_\Xx k_x d\xi(x),$$
		where  the integral   should be understood as a  Bochner integral.
	\end{example}

	\begin{remark}\label{rem:est}  (1)   In classical statistics, see  e.g.  \cite[\S 13, p. 51]{Borovkov1998}, \cite[p.4]{IH1981}, \cite[\S 4, p. 82]{AN2000}, \cite[Definition 5.1, p. 277]{AJLS2017},  one considers  only parameter estimations   for  parameterized   statistical models. In this case, an   estimator
		is a map from   $\Xx$ to  the    parameter  set $\Theta$ of
		a  statistical   model  $\pb (\Theta)\subset \Pp(\Xx)$.   Usually  one assumes that   the parametrization $\pb: \Theta \to \pb (\Theta)$  is 1-1,   hence,  a parameter  estimation is  equivalent to   a nonparametric  estimation  in the sense of Definition  \ref{def:est}.  Note     that  the
		ultimate aim of  a statistical experiment  is to estimate  the
		probability  measure  generating  the  observable of   the experiment. In general,  we can  only assume  that   the  unknown generating  probability measure   belongs to   a   statistical  model  $P_\Xx \subset \Pp(\Xx)$. In  this case, we
		need  to  use  non-parametric  estimation, see e.g. \cite[p. 1]{Tsybakov2009}.  Note that,  by Example \ref{ex:para},    $P_\Xx$    has  a natural  structure  of a $C^1$-diffeological  statistical model.

		(2)  The notion   of a $\varphi$-estimation    occurs  in classical    statistics in similar fashion,   see e.g. \cite[p. 52]{Borovkov1998},`  where  the author  called   similar estimators  {\it  substitution  estimators},  and in 
		\cite[Definition 1.2, p. 4]{LC1998}, where   the  authors  consider    {\it  estimands}, which  are  versions  of
		$\varphi$-estimators  for a  parameter  estimation  problem,
		see \cite[p. 279]{AJLS2017}.
		
	\end{remark}

	For   $\varphi \in Map(P_\Xx, V)$ and $l \in V'$ we    denote by $\varphi^l$ the composition  $l \circ \varphi$.  Then  we set
	$$L^2_{\varphi}(\Xx, P_\Xx): = \{\hat\sigma: \Xx \to P_\Xx|\: \varphi^l \circ \hat \sigma \in L^2 _\xi (\Xx) \text{ for all } \xi \in P_\Xx \text{  and } l \in V'\}.$$
	
	For   $\hat \sigma \in  L^2_\varphi(\Xx, P_\Xx)$  we define  the $\varphi$-mean value    of $\hat  \sigma$, denoted  by
	$\varphi_{\hat \sigma}: P_\Xx \to  V'' $,  as follows (cf. \cite[(5.54), p. 279]{AJLS2017})
	$$\varphi_{\hat \sigma}   (\xi) (l): = \E_\xi (\varphi^l \circ \hat \sigma) \text{  for }   \xi \in P_\Xx \text{ and } l \in V'. $$
	Let  us identify   $V$ with a  subspace   in $V^{''}$ via the canonical pairing.
	
	The  difference  $b^{\varphi}_{\hat{\sigma}}: = \varphi_{\hat{\sigma}}- \varphi \in Map(\Pp_\Xx,  V ^{''})$   will be called {\it the bias }  of the $\varphi$-estimator $\hat \sigma_{\varphi}$.
	
	For all $\xi \in P_\Xx$ we   define  a    quadratic function $MSE^{\varphi}_\xi [\hat \sigma]$     on $V'$, which   is 
	called  the {\it mean  square error  quadratic   function}  at $\xi$,   by setting  for $l, h \in V'$ (cf. \cite[(5.56), p. 279]{AJLS2017})
	\begin{equation}
	\label{eq:qd}
	MSE^\varphi _\xi[\hat \sigma] (l, h): = \E_\xi [(\varphi^l\circ \hat \sigma(x)-\varphi^l(\xi))\cdot  (\varphi^h\circ \hat \sigma(x)-\varphi^h(\xi))].
	\end{equation}
	
	Similarly we define   the {\it variance  quadratic  function}  of   the $\varphi$-estimator $\varphi\circ \hat \sigma$  at $\xi \in \Pp_\Xx$
	is     the quadratic  form $V^\varphi_\xi [\hat \sigma]$ on $V'$ such that for all   $l, h \in V'$ we have (cf.  \cite[(5.57), p.279]{AJLS2017})
	$$V^\varphi_\xi [\hat \sigma](l, h) = \E_\xi [\varphi^l\circ \hat \sigma(x)-E_\xi (\varphi^l\circ \hat \sigma(x))\cdot  \varphi^h\circ \hat \sigma(x)-E_\xi(\varphi^h\circ \hat \sigma(x))].$$
	Then it is known that \cite[(5.58), p. 279]{AJLS2017}
	\begin{equation}\label{eq:vmse}
	MSE^\varphi_\xi [\hat \sigma] (l, h) = V^\varphi_\xi [\hat \sigma](l,h) + \la b^\varphi_{\hat \sigma}(\xi) , l\ra  \cdot \la  b^\varphi_{\hat \sigma} (\xi), h\ra.
	\end{equation}
	
	\begin{remark}\label{rem:mse}  Assume  that $V$ is a  real Hilbert    space
		with  a scalar  product $\la \cdot, \cdot \ra$ and the  associated   norm $\|\cdot \|$.  Then   the  scalar  product  defines  a   canonical  isomorphism  $V = V', \:   v (w): = \la  v, w\ra$  for all $v, w \in V$.
		For  $\hat \sigma \in L^2_\varphi (\Xx, P_\Xx)$ 
		the mean square  error $MSE_\xi^\varphi(\hat{\sigma})$  of  the $\varphi$-estimator  $\varphi \circ \hat \sigma$ is defined 
		by
		\begin{equation}
		\label{eq:mse}
		MSE_\xi^\varphi (\hat \sigma): = 
		\E_\xi (\| \varphi \circ \hat \sigma  - \varphi(\xi)\| ^2).
		\end{equation}
		The   RHS  of (\ref{eq:mse}) is well-defined, since $\hat \sigma  \in L^2 _\varphi (\Xx, P_\Xx)$ and  therefore    
		$$\la \varphi\circ \hat  \sigma(x), \varphi \circ \hat \sigma (x) \ra  \in L^1 (\Xx, \xi) \text{
			and }\la \varphi \circ \hat  \sigma(x), \varphi(\xi) \ra \in  L^2 (\Xx, \xi).$$
		Similarly, we define    the variance  of a $\varphi$-estimator
		$\varphi\circ \hat \sigma$  at $\xi$   as follows
		$$V^{\varphi}_\xi (\hat \sigma): = \E_\xi (\|\varphi \circ  \hat \sigma - \E_\xi (\varphi \circ \hat \sigma)\| ^2).$$
		
		If  $V$  has  a countable  basis   of orthonormal  vectors $v_1, \cdots,
		v_\infty$, then  we  have
		\begin{equation}
		\label{eq:mse2}
		MSE^\varphi_\xi (\hat \sigma) = \sum_{i =1}^\infty MSE ^\varphi_\xi[\hat \sigma](v_i, v_i),
		\end{equation}
		\begin{equation}
		\label{eq:var2}
		V^\varphi_\xi (\hat \sigma) = \sum_{i =1}^\infty  V^\varphi_\xi[\hat \sigma](v_i, v_i).
		\end{equation}
	\end{remark}
	
	\
	
	Now  we assume that  $(P_\Xx, \Dd_\Xx)$ is  an  almost 2-integrable
	$C^k$-diffeological   statistical model.    For any $\xi \in P_\Xx$  
	let  $ T_{\xi}^\g (P_\Xx,\Dd_\Xx)$  be the completion of  $T_\xi (P_\Xx,\Dd_\Xx)$ w.r.t.  the     Fisher  metric $\g$. Since  $T_\xi ^\g (P_\Xx, \Dd_\Xx)$  is   a Hilbert  space,  the map  
	$$L_\g: T_\xi^\g (P_\Xx,\Dd_\Xx) \to (T_\xi^\g (P_\Xx,\Dd_\Xx))',\, L_\g (v)(w) := \la  v, w\ra_\g,  $$   is an isomorphism.
	Then we define  the  inverse $\g ^{-1}$ of the Fisher  metric $\g$ on $ (T_\xi^\g (P_\Xx, \Dd_\Xx))'$ as  follows
	\begin{equation}
	\label{eq:Finv}
	\la   L_g  v, L_g  w \ra _{\g^{-1}}: = \la  v, w\ra_\g
	\end{equation}
	
	\begin{definition}
		\label{def:reg}(cf. \cite[Definition 5.18, p. 281]{AJLS2017})
		Assume  that $\hat \sigma \in L^2_{\varphi}(\Xx, P_\Xx)$.  We shall call  $\hat \sigma$  {\it a $\varphi$-regular estimator},
		if for all  $l \in V'$   the function  $\xi \mapsto \|\varphi^l \circ  \hat \sigma\|_{L^2 (\Xx, \xi)}$ is locally bounded, i.e., 
		for all  $\xi_0 \in P_\Xx$ 
		$$\lim_{\xi \to \xi_0} \sup \|  \varphi^l\circ \hat \sigma  \|_{L^2 (\Xx, \xi)}  < \infty.$$
	\end{definition}

	\begin{proposition}\label{prop:reg}  Assume that   $(P_\Xx, \Dd_\Xx)$  is a  2-integrable $C^k$-diffeological statistical model,  $V$ is   a topological vector  space,  $\varphi \in Map (P_\Xx, V)$  and $\hat \sigma: \Xx \to P_\Xx$  is a $\varphi$-regular  estimator. Then  the  $V''$-valued function  $\varphi_{\hat \sigma}$  is  Gateaux-differentiable  on $(P_\Xx, \Dd_\Xx)$.  Furthermore  for any $l' \in V'$  the differential $d\varphi^l_{\hat \sigma}(\xi)$  extends to   an element  in $(T_\xi^\g(P_\Xx,\Dd_\Xx))'$  for all $\xi \in P_\Xx$.
		\label{prop:nabla1}
	\end{proposition}
	
	\begin{proof} Assume that  a  map 
		$c: \R \to P_\Xx$  belongs  to $\Dd_\Xx$.
		Then  $(\R, \Xx, c)$ is a  2-integrable parametrized   statistical model.  By Lemma  5.2 in \cite[p. 282]{AJLS2017}   the composition
		$\varphi_{\hat \sigma} \circ c$  is differentiable.  This proves the first  assertion of Proposition \ref{prop:reg}.
		
		Next  we shall  show  that  $d\varphi_{\hat \sigma}(\xi)$ extends  to an element  in  $(T_\xi^\g(P_\Xx, \Dd_\Xx))'$ for
		all $\xi  \in P_\Xx$. Let  $X \in  C_\xi (P_\Xx, \Dd_\Xx)$  and  $c: \R \to P_\Xx$    be a $C^k$-curve   such that $c(0) = \xi$ and $\dot c (0)= X$.   By  Lemma 5.3  \cite[p. 284]{AJLS2017}  we have
		\begin{equation}\label{eq:log1}
		\p_X (\varphi^l_{\hat \sigma})= \int_\Xx (\varphi ^\l \circ \hat \sigma  (x)  -\E_\xi (\varphi^l \circ \hat \sigma)\cdot  \log X\, d\xi(x),
		\end{equation}
		where $\varphi ^\l \circ \hat \sigma  (x)  -\E_\xi (\varphi^l \circ \hat \sigma) \in L^2 (\Xx, \xi)$.   Denote by  $\Pi_\xi: L^2(\Xx, \xi) \cdot \xi \to  T^\g_{\xi} P_\Xx$   the orthogonal   projection. Set
		\begin{equation}\label{eq:grad}
		\mathrm{grad}_g  (\varphi^l_{\hat \sigma} ): =  \Pi_\xi [(\varphi ^\l \circ \hat \sigma  (x)  -\E_\xi (\varphi^l \circ \hat \sigma) )\cdot \xi] \in  T_\xi^\g P_\Xx.
		\end{equation}
		Then  we  rewrite  (\ref{eq:log1}) as follows
		$$  \p_X (\varphi^l) = \la   \mathrm{grad}_g  (\varphi^l_{\hat \sigma} ),  X \ra _\g.$$
		Hence  $d\varphi ^l_{\hat \sigma}$ is the restriction of $L_\g (\mathrm{grad}_\g  (\varphi^l_{\hat \sigma} ))  \in (T_\xi^\g(P_\Xx,\Dd_\Xx))'$. This  completes  the proof   of Proposition
		\ref{prop:reg}.
	\end{proof}
	
	For any $\xi \in \Pp_\Xx$  we 
	denote  by  $(\g ^\varphi_{\hat \sigma}) ^{-1} (\xi)$  to be   the   following  quadratic form    on $V'$:
	\begin{equation}\label{eq:g1}
	(\g ^\varphi_{\hat \sigma})^{-1} (\xi)(l, k) : = \la d\varphi ^l_{\hat \sigma}, d\varphi^k_{\hat \sigma} \ra _{\g ^{-1}}(\xi): = \la \mathrm{grad}_\g  (\varphi^l_{\hat \sigma} ), \mathrm{grad}_\g  (\varphi^k_{\hat \sigma} )\ra .
	\end{equation}
	
	\begin{theorem}[Diffeological  Cram\'er-Rao inequality] 
		\label{thm:cr}  Let $(P_\Xx, \Dd_\Xx)$   be a   2-integrable   $C^k$-diffeological statistical model, $\varphi$  a $V$-valued function on $P_\Xx$ and  $\hat \sigma \in L^2_{\varphi} (\Xx, P_\Xx)$  a $\varphi$-regular  estimator.  Then the  
		difference $\V_{\xi}^\varphi[\hat \sigma] -   (\hat \g^{\varphi}_{ \hat \sigma})^{-1} (\xi)$     is   a  positive semi-definite  quadratic form on $V'$  for any $\xi \in P_\Xx$. 
	\end{theorem}
	
	\begin{proof}  To prove Theorem  \ref{thm:cr} it suffices  to show that for any  $l \in V'$ we have
		\begin{equation}\label{eq:cr1}
		\E_\xi  (\varphi^l\circ \hat \sigma  - \E_\xi (\varphi ^l \circ  \hat\sigma))^2 \ge  \|\mathrm{grad}_\g  (\varphi^l_{\hat \sigma} ))\|^2_\g.
		\end{equation}
		Clearly  (\ref{eq:cr1})    follows from (\ref{eq:grad}).  This completes  the proof of Theorem \ref{thm:cr}.
	\end{proof}

	Theorem \ref{thm:cr} is an extension  of  the general Cram\'er-Rao inequality \cite[Theorem 2]{LJS2017b}, see also \cite[Theorem 5.7, p. 286]{AJLS2017}.
	\section{Discussion}\label{sec:dis}
	The extension of   the notion
	of  a $k$-integrable  parametrized measure  model introduced  in \cite{AJLS2015, AJLS2018}, see  also \cite{AJLS2017},  to   the  notion of
	an almost  $k$-integrable     diffeological measure  model  can  be done in    the same way.

	(2)  There  are  two  main  differences  between  parameterized  statistical  models  and    $C^k$-diffeological     statistical models.  Firstly,   the parameter  space  of a  parameterized      statistical model is a {\it single  smooth  Banach  manifold} and
	{\it parameter  spaces}  for  a   $C^k$-diffeological   statistical model   can be  {\it different  but compatible}. Secondly, 
	parameter
	spaces for a  $C^k$-diffeological  statistical  model     are  {\it finite  dimensional}. If $k = \infty$, this assumption is well-motivated \cite{IZ2013}, see also Remark \ref{rem:diffeo} (2).

	(3) It would be  interesting  to apply  theory of    $C^k$-statistical  models   to    stochastic  processes.   It  is known that  Banach manifolds  are not suitable   for many  question of global analysis, see  e.g., \cite[p. 1]{KM1997}, and therefore,   theory  of parameterized  measure  models  might have  limited   applications to    stochastic  processes.   On the other hand, we would like to  notice  that  there are    many open questions  in        theory of $C^\infty$-diffeological spaces,  e.g.,    we don't   know   under which  condition we   can define
	the Levi-Civita  connection   on  a Riemannian $C^\infty$-diffeological  space. Furthermore, theory  of $C^k$-diffeological
	spaces   has not been  considered  before, if $k \not = \infty$.

	(4) The        variational  calculus  founded by Leibniz and Newton   is  a cornerstone   of 
	differential geometry and  modern  analysis. In our opinion  it   is best    expressed in the language of   diffeological spaces  that  declares  which      mappings into a diffeological space  are smooth. This language   is  a   counterpart   of the  language of  ringed spaces  in algebraic geometry  that  declares  which   functions are  algebraic.

	\vspace{6pt}

	\section*{Acknowledgement}
	The author   would like to thank Patrick  Iglesias-Zemmour   for  a stimulating discussion on diffeology,   Lorenz Schwachh\"ofer   for helpful  comments on  an early version
		of this  paper  and Tat  Dat To for suggesting to consider      Friedrich's  examples  in \cite{Friedrich1991}.  A part  of this paper has been done  during the  Workshop ``Information Geometry" in Toulouse  October 14-18, 2019. The author  would like to thank the organizers  and especially  Stephane
		Puechmorel for their invitation  and hospitality  during the  workshop.  The author  is grateful to   anonymous  referees  for their     critical   comments  and suggestions,     which  helped her to   improve  significantly the exposition  of    this  paper.


\begin{thebibliography}{999}
		\bibitem[AJLS2015]{AJLS2015} Ay N., Jost J., L\^e H.V., Schwachh\"ofer  L.,
		Information geometry and sufficient statistics, {\em Probability Theory and Related Fields} {\bf 2015}, {\em  162}, 327--364. 
		\bibitem[AJLS2017]{AJLS2017} Ay N., Jost J., L\^e H.V., Schwachh\"ofer  L., {\em Information geometry},  Springer  Nature: Cham,  Switzerland, 2017.
		\bibitem[AJLS2018]{AJLS2018} Ay N., Jost J., L\^e H.V., Schwachh\"ofer  L., Parametrized measure models,     {\em Bernoulli} {\bf 2018}, {\em  24}, 1692--1725.
		\bibitem[Amari1985]{Amari1985} S. Amari S., {\em Differential-Geometric Methods in Statistics}, Lecture Notes in Statistics 28, Springer-Verlag: Heidelberg, Germany,  1985.
		\bibitem[Amari2016]{Amari2016} Amari S.,  {\em Information Geometry and Its Applications}, Applied Mathematical Sciences, vol. 194, Springer: Berlin, Germany, 2016.
		\bibitem[AN2000]{AN2000} Amari S.,  and  Nagaoka H., {\em Methods of Information Geometry}. Translations of Mathematical
		Monographs 191,  Amer. Math. Soc.: Providence, RI, USA, 2000.
		\bibitem[Bogachev2018]{Bogachev2018} Bogachev V.I., {\em  Weak convergence  of measures}, Mathematical Surveys and Monographs, vol. 234, Amer. Math. Soc.:  Providence, RI, USA, 2018.
		\bibitem[Borovkov1998]{Borovkov1998}  Borovkov  A.A., {\em Mathematical statistics}, Gordon and Breach Science Publishers: Amsterdam, The Nethelands, 1998.
		\bibitem[Chen1977]{Chen1977}  Chen K.T.,  Iterated path integrals, {\em Bull. Amer. Math. Soc.},  {\bf 1977}, {\em 83}, 
		pp. 831--879.
		\bibitem[Chentsov1972]{Chentsov1972} Chentsov N. , {\em Statistical decision rules and optimal inference}, Nauka: Moscow,  Russia,   1972,   English translation  in:  Translation of Math. Monograph vol. 53, Amer. Math. Soc.:  Providence, RI, USA, 1982.
		\bibitem[Friedrich1991]{Friedrich1991} Friedrich T.,  Die Fisher-Information  und symplektische   Strukturen, {\em Math. Nachr.}  {\bf 1991}, {\em 153}, 273--296.
		\bibitem[Giry1982]{Giry1982}{\sc M. Giry},  A categorical approach to probability theory, In: B. Banaschewski, editor, {\em Categorical Aspects of Topology and Analysis}, Lecture Notes in Mathematics vol.   915, pp. 68--85, Springer: Berlin- Heidelberg, Germany, 1982.
		\bibitem[Grabiner1974]{Grabiner1974} Grabiner S.,  Range of products of operators,
		{\em Canadian J.of Math.} {\bf 1974}, {\em   XXVI}, 1430--1441. %
		\bibitem[IH1981]{IH1981} Ibragimov I.A.  and  Has'minskii R Z., {\em Statistical Estimation: Asymptotic Theory},  Springer-Verlag: New-York, USA,  1981.
		\bibitem[JLS2017]{JLS2017}Jost  J. , L\^e  H.V., and Schwachh\"ofer L.,   Cram\'er-Rao inequality on singular statistical models I, {\bf 2017}, {\em arXiv}:1703.09403.
		\bibitem[JLLT2019]{JLLT2019} Jost  J. , L\^e H.V., Luu D.H. and  Tran T.D.,  Probabilistic mappings and Bayesian nonparametrics, {\bf 2019}, {\em arXiv}:1905.11448.
		\bibitem[IZ2013]{IZ2013} Iglesias-Zemmour P., {\em Diffeology}, Amer. Math. Soc.:  Providence, RI, USA, 2013.
		\bibitem[KM1997]{KM1997} Kriegl A.  and	 Michor P. W., {\em The Convenient  Setting of
			Global Analysis}, Amer. Math. Soc.:  Providence, RI, USA,  1997.
		\bibitem[Lawvere1962]{Lawvere1962} Lawvere W.F.,   The category of probabilistic mappings, {\bf 1962}.{\em Unpublished,
			Available at } https://ncatlab.org/nlab/files/lawvereprobability1962.pdf.
		\bibitem[LC1998]{LC1998} Lehmann E. L. and   Casella G.,{\em Theory of Point Estimation}, 2nd Edition,  Springer-Verlag: New York, USA, 1998.
		\bibitem[LJS2017b]{LJS2017b} L\^e H.V., Jost J., Schwachh\"ofer L., The Cram\'er-Rao Inequality on Singular Statistical Models, In: {\em  Proceedings of Conference  ``Geometric  Science of Information"},  GSI 2017, Paris  November 7-9, 2017,   LNCS  vol. 10589, pp. 552-560, Springer  Nature: Cham,  Switzerland,  2017.
		\bibitem[LSV2013]{LSV2013} L\^e  H.V., Somberg P. and  Van\v zura J.,  Smooth structures on pseudomanifolds  with isolated conical singularities, {\em Acta Mathematica Vietnamica},  {\bf 2013},{\em 38}, 33--54.
		\bibitem[LSV2015]{LSV2015}   L\^e  H.V., Somberg P. and  Van\v zura J.,  Poisson smooth structures on stratified symplectic spaces,  In:{\em The Springer Proceedings in Mathematics \& Statistics ``Mathematics in the 21st
			Century, 	6th World Conference, Lahore, March 2013} Volume 98, chapter 7,  pp 181--204, Springer-Verlag: Basel, Switzerland, 2015.
		\bibitem[MS1966]{MS1966} Morse N. and   Sacksteder R., Statistical isomorphism, {\em Annals of Math. Statistics},  {\bf 1966}, {\em 37}, 203--214.
		\bibitem[McCullagh2000]{McCullagh2000}McCullagh P., What is a statistical model, {\em The Annals of Statistics}{\bf 2002}, {\em 30} 1225--1310.
		\bibitem[MFSS2017]{MFSS2017}  Muandet K. ,  Fukumizu K.,  Sriperumbudur B. and	 Sch\"olkopf B.,  Kernel Mean Embedding of Distributions: A Review and Beyonds, {\em Foundations and Trends in Machine Learning}, {\bf 2017}, {\em  10},: No. 1-2, pp 1--141.
		\bibitem[PS1995]{PS1995} Pistone G. and  Sempi C.,  An infinite-dimensional structure on the space of all the probability measures equivalent to a given one.{\em Ann. Stat.}
		{\bf 1995}, {\em 23}, 1543--1561.
		\bibitem[Souriau1980]{Souriau1980}Souriau J.-M., Groupes diff\'erentiels, {\em  Lect. Notes in Math., vol. 836}, Springer Verlag, pp. 91--128, 1980.
		\bibitem[Schervish1997]{Schervish1997} Schervish M. J.,{\em  Theory of Statistics}, Springer-Verlag: New York, USA, 2nd Edition, 1997.	
		\bibitem[Tsybakov2009]{Tsybakov2009} Tsybakov A.B., {\em Introduction to Nonparametric Estimation}, Springer Science+Business Media: New York, USA, 2009.
\end{thebibliography}
\end{document}